\documentclass[a4paper,oneside]{article}

\usepackage{amssymb,amsmath,amsthm,amscd}
\usepackage{mathrsfs}

\usepackage[T1]{fontenc}
\usepackage[utf8]{inputenc}

\usepackage[english]{babel}

\usepackage{geometry}
\usepackage{hyperref}
\usepackage{graphicx} 
\usepackage{float}

\newcommand{\Hyp}{\mathbb{H}}

\newcommand{\R}{\mathbb{R}}
\newcommand{\D}{\mathbb{D}}

\newcommand{\G}{\Gamma}
\newcommand{\g}{\gamma}

\newcommand{\tv}{\rightarrow}

\newcommand{\scal}[2]{\left\langle {#1} \middle| {#2} \right\rangle}
 
\newcommand{\fonction}[5]{\begin{array}{l|ccl}
{#1} : & {#2} & \longrightarrow & {#3} \\
        & {#4} & \longmapsto & {#5} 
    \end{array}
    }

\newtheorem{theorem}{Theorem}[section]
\newtheorem{lemme}[theorem]{Lemma}
\newtheorem{proposition}[theorem]{Proposition}
\newtheorem{Corollaire}[theorem]{Corollary}

\newtheorem{definition}[theorem]{Definition}

\DeclareMathOperator{\Leb}{Leb}

\DeclareMathOperator{\Isom}{Isom}
\DeclareMathOperator{\Card}{Card}
\DeclareMathOperator{\PSL}{PSL}

\DeclareMathOperator{\vol}{Vol}
\DeclareMathOperator{\Vol}{Vol }
\DeclareMathOperator{\Met}{Met }

\newtheorem*{corollary*}{Corollary}
\newtheorem*{theorem*}{Theorem}

\title{Entropy of embedded surfaces in quasi-fuchsian manifolds}
\author{Olivier Glorieux}

\begin{document}

\maketitle
\begin{abstract}
We compare critical exponent for quasi-Fuchsian groups acting on the hyperbolic 3-space and entropy of invariant disks embedded in $\Hyp^3$. We give a rigidity theorem for all embedded surfaces when the action is Fuchsian and a rigidity theorem for negatively curved surfaces when the action is quasi-Fuchsian. 
\end{abstract}
 \section{Introduction}
The aim of this paper is to compare two geometric invariants of Riemannian manifolds: critical exponent and volume entropy. The first one is defined through the action of the fundamental group on the universal cover, the second one is defined for compact manifolds as the exponential growth rate of the volume of balls in the universal cover. These two invariants have been studied in many cases, we pursue this study for quasi-Fuchsian manifolds. 

Let $\G$ be a group acting on a simply connected Riemannian manifold $(X,g)$. If the action on $X$ is discrete we define the \emph{critical exponent }  by 
\begin{eqnarray}\label{def,eq - critical exponent}
\delta(\G) := \limsup_{R\tv \infty} \frac{1}{R} \Card \{ \g\in \G \, | \, d(\g\cdot o , o) \leq R\},
\end{eqnarray}

where $o$ is any point in $X$. It does not depends  on this particular base point thanks to triangle inequality. If we want to insist on the space on which $\G$ acts we will write $\delta(\G, X)$. 

The volume entropy $h(g)$ of a Riemannian compact manifold $(\Sigma, g)$  is defined by 
\begin{eqnarray}\label{def-eq - volume entropy}
h(g) := \lim_{R\tv \infty} \frac{\log \Vol_g( B_g(o,R))}{R},
\end{eqnarray}
where $B_g(o,R)$ is the ball of radius $R$ and center $o$ in the universal cover of $\Sigma$. We will also use the notation $h(X)$ for simply connected manifolds  $X$ as the exponential growth rate of its balls. 

It is a classical fact, using a simple volume argument that the volume entropy coincides with the critical exponent of $\pi_1(\Sigma)$ acting on $\widetilde{\Sigma}$. Moreover, a famous theorem of G. Besson, G. Courtois and S. Gallot \cite{besson1995entropies} said that the entropy allows to distinguish hyperbolic metric in the set of all metrics, $\Met(\Sigma)$. Remark that entropy is sensitive to homothetic transformations : for any $\lambda>0$ we have $h(\lambda^2 g) =\frac{1}{\lambda} h(g)$.  Assume that $\Sigma$ admits an hyperbolic metric $g_0$ and let $\Met_0(\Sigma)$ be the set of metrics on $\Sigma$ whose volume is equal to $\Vol(\Sigma,g_0)$, then Besson, Courtois, Gallot's Theorem says for all $g\in \Met_0(\Sigma) $  : 
\begin{eqnarray}\label{eq -BCG}
h(g)\geq h(g_0).
\end{eqnarray}
with equality if and only if $g=g_0$. 
\\

%
%

Our aim is to study the behavior of the volume entropy for a subset of all the metrics on a surface.  This subset is the metrics induced by an incompressible embedding into a quasi-Fuchsian manifolds. It has not the cone structure of $\Met(\Sigma)$ : it is not invariant by all homothetic transformations. Hence we will look at the behavior of $h(g)$ without normalization by the volume. 
\\

Let $S$ be a compact surface of genus $g\geq 2$ and $\G=\pi_1(S)$ its fundamental group. A Fuchsian representation of $\G$ is a  faithful and discrete representation in $\PSL_2(\R)$. A quasi-Fuchsian representation is a perturbation of Fuchsian representation in $\PSL_2(\mathbb{C})$. More precisely it is a discrete and faithfull representation of $\G$  into $\Isom(\Hyp^3)$, such that the limit set on $\partial \Hyp^3$ is a Jordan curve.  A celebrated theorem of R. Bowen \cite{bowen1979hausdorff}, asserts that for quasi-Fuchsian representations, critical exponent is minimal and  equal to $1$ if and only if the representation is Fuchsian.

We choose an isometric, totally geodesic embedding of $\Hyp^2$ in $\Hyp^3$ (The equatorial plane in the ball model for example). This embedding gives a inclusion 
$i : \Isom(\Hyp^2)\tv \Isom(\Hyp^3)$. 

Let $\rho$ be a Fuchsian representation of $\G$. 
The group $\G$ acts naturally on $\Hyp^2$, respectively $\Hyp^3$, by $\rho$, respectively $i\circ \rho$. For every points $o\in \Hyp^2$ we have
$$d_{\Hyp^3} (i\circ \rho(\g) o , o ) =d_{\Hyp^2} (\rho(\g) o,o),$$
since $\Hyp^2$ is totally geodesic in $\Hyp^3$. The critical exponent for these two actions of $\G$ are then equal 
$$\delta(\G , \Hyp^3) =\delta(\G ,\Hyp^2) =1.$$

In light of this trivial example, two questions rise up. What is the entropy of a $\G$ invariant disk which is not totally geodesic ? What happens when we modify the Fuchsian representation in $\PSL_2(\mathbb{C})$ ?

We will answer to the first question. Since $ \rho$ is a Fuchsian representation, the critical exponent of $\G$ acting on $\Hyp^3$ through $i\circ \rho$ is $1$, and we have the following 

\begin{theorem}\label{th - main fuchsian}
Suppose $\G$ is Fuchsian. Let $\Sigma$ be a $\G$ invariant disk embedded in $\Hyp^3$. We have 
\begin{eqnarray}
h(\Sigma)  \leq \delta(\G,\Hyp^3),
\end{eqnarray} 
equality occurs if and only if $\Sigma$ is the totally geodesic hyperbolic plane preserved by $\G$.
\end{theorem}
Remarks that $\delta(\G,\Hyp^3) = h(\Sigma,g_0)$, hence the last theorem can be rewritten as follow :  
\begin{theorem}
For all metrics $g$ obtained as induced metrics by an incompressible embedding in a Fuchsian manifold we have 
\begin{eqnarray}\label{eq - inverse de BCG sans normalisation}
h(g)\leq h(g_0)
\end{eqnarray}
with equality if and only if $g=g_0$. 
\end{theorem}
We \emph{did not} renormalize by the volume, this explains the dichotomy between (\ref{eq -BCG}) and (\ref{eq - inverse de BCG sans normalisation}).
\\

We will prove this theorem in the next section. The inequality is trivial since the induce distance between two points is always greater than the distance in $\Hyp^3$ : $d_\Sigma \geq d_{\Hyp^3} $, but the rigidity is not. We have no geometrical (curvature) hypothesis on $\Sigma$, therefore it is not obvious at all to show that the inequality is strict as soon as $\Sigma$ is not totally geodesic. Indeed we cannot use the "usual" techniques of negative curvature like Bowen-Margulis measure, or even the uniqueness of geodesic between two points. 

We obtain an answer to the second question under a geometrical hypothesis on the curvature: 
\begin{theorem}\label{th - main qf}
Let $\G$ be a quasi-Fuchsian group and $\Sigma\subset \Hyp^3$  a $\G-$invariant embedded disk. We suppose that $\Sigma$ endowed with the induced metric has negative curvature. We then have 
$$h(\Sigma)  \leq I(\Sigma,\Hyp^3)\delta(\G, \Hyp^3),$$
where $I(\Sigma,\Hyp^3)$ is the geodesic intersection between $\Sigma$ and $\Hyp^3$. Moreover, equality occurs if and only if the length spectrum of $\Sigma/\G$ is proportional to the one of $\Hyp^3/\G$. 
\end{theorem}

The \emph{geodesic intersection} will be defined in section \ref{sec - geodesic intersection}. Roughly, it is the average ratio of the length between two points of $\Sigma$ for the extrinsic and intrinsic distance. We need the curvature assumption to define and use this invariant.

This Theorem implies Theorem \ref{th - main fuchsian} only for negatively curved embedded disks but not in its full generality. Indeed, when $\G$ is Fuchsian, and $\Sigma/\G$ has the same length spectrum as $Hyp^3/\G$ it follows directly by the work of J-P Otal \cite{otal1990spectre} that $\Sigma =\Hyp^2/\G$. However, using the fact that $\Sigma$ is embedded in $\Hyp^3$ we will be able to prove without the Fuchsian hypothesis that if the two length spectrum are equal then $\Sigma$ is totally geodesic, and therefore we obtain the following corollary of Theorem \ref{th - main qf}:
\begin{Corollaire}\label{cor - h(sigma)=delta imply fuchsian}
Under the assumptions  of Theorem \ref{th - main qf} we have 
$$h(\Sigma) \leq \delta(\G,\Hyp^3),$$ 
with equality if and only if  $\G$ is fuchsian and $\Sigma$ is the totally geodesic hyperbolic plane, preserved by $\G$. 
\end{Corollaire}

The proof of this corollary raises the  following question generalizing this result: if a quasi-Fuchsian manifold has the same length spectrum as a negatively curved surface, does it implie that it is in fact Fuchsian ? We answer this question using a well known result pf Y. Benoist, showing the following theorem

\begin{theorem}\label{th - spectre proportionel et constant curvature implique fuchsian}
Let $M$ be a quasi-Fuchsian manifold and $\Sigma$ a hyperbolic (in the sense that it has constant curvature equal $-1$) surface. Suppose that $M$ and $\Sigma$ have proportional length spectrum (ie. there exists $k\in \R^+$ such that for all $\g \in \G, \, \ell_M(\g)=k \ell_\Sigma(\g)$), then $M$ is Fuchsian and $\Sigma$ is isometric to the totally geodesic surface in $M$. 
\end{theorem}

Theorem \ref{th - main qf} has to be compared to results obtained by G. Knieper who compared entropy for two different metrics on the same manifolds and our proof of Theorem \ref{th - main qf} follows his paper \cite{knieper1995volume}. 
As in his paper, we obtain that the intersection is larger than $1$ as soon as $\G$ is not Fuchsian. 

It is also related to the work of M. Bridgeman and E. Taylor \cite{bridgeman2000length}, indeed we answer by the negative to Question 2 of their paper.  And finally, we can see our work as an extension of U. Hamenstadt's paper \cite{hamenstadt2002ergodic}, where she compared  the geodesic intersection between the boundary of convex hull and $\Hyp^3$ for quasi-Fuchsian manifolds.

As we said, the two proofs are very different one from each others. For the Fuchsian case, we give precise estimates for the length of some paths of the hyperbolic plane. We show that in some sense the length between two points on $\Sigma$ is much greater than the extrinsic distance between those two points. For quasi-Fuchsian manifolds, we use well known techniques of negative curvature geometry: we compare the Patterson Sullivan measures for $\Hyp^3$ and for $\Sigma$. 

\paragraph{Acknowledgements} We want to thank Maxime Wolff for his help in the proof of Theorem \ref{th - main fuchsian}, and the referee for usefull comments concerning rigidity questions. 

\section{Fuchsian case}
In this section we are going to prove Theorem \ref{th - main fuchsian}. This theorem has a strong condition on $\G$, ie. it is conjugate to a subgroup of $\PSL_2(\R)$ but we make no  geometrical  assumptions on $\Sigma$. As we said, there could be more than one geodesic between two points on $\Sigma$. 

We already remarked that the inequality is trivial, as is the equality when $\Sigma$ is totally geodesic. Therefore, the only thing left to prove is the strict inequality when $\Sigma$ is not totally geodesic or in other words if $\Sigma\neq \Hyp^2$ then $h(\Sigma) < 1$. 

The proof of the theorem is based on the comparison between the distances on equidistant surfaces of  the totally geodesic  $\G$-invariant hyperbolic plane. We are going to prove several lemmas which together gives Theorem \ref{th - main fuchsian}. The strict inequality follows directly from Lemmas `\ref{lem - h(sigma)<delta(D,dm)} and \ref{lem- delta dm < 1-delta}. We denote by $\D$ the totally geodesic, $\G$-invariant plane. The induced metric on  $\D$ is the usual hyperbolic metric, and we will denote it by $\Hyp^2$. We are first going to see that between all the equidistant surfaces, $\Hyp^2$ has the biggest entropy. Then we will make this argument work when only one part of the surface is "above" $\D$. The idea to prove it, is to consider another distance $d_m$ on $\D$, which will be used as an intermediary between $\Sigma$ and $\Hyp^2$. We will explain, after the definition of $d_m$ how the two comparisons will be proved. 

Let us begin to parametrize $\Hyp^3 $ by $\Hyp^2\times \R $ as follows: take an orientation for the unit normal tangent space of $\Hyp^2$, then to a point $x\in \Hyp^3$ we associate $s(x)$ the orthogonal projection from $\Hyp^3$ to $\Hyp^2$. It is the first parameter of the parametrization. The oriented distance along this geodesic gives the second one.  Hence the parametrisation, called Fermi coordinates, is defined by 
$$\begin{array}{ccc}
\Hyp^3 &\mapsto &  \Hyp^2\times \R	\\
z & \tv &  (s(z), \hat{d}(z,s(z)) )
\end{array}$$
 where $\hat{d}$ is the oriented distance defined by the choice of the orientation on the unit normal tangent of $\Hyp^2$.  With this parametrization, the metric on $\Hyp^3$ is 
$$g_{\Hyp^3} = \cosh^2(r)  g_0 + dr^2.$$

Look at $S(r)$ the equidistant disk  at distance $r$ of $\Hyp^2$, its metric, induces by the one on $\Hyp^3$, is $g_r =\cosh^2(r) g_0$. It is  isometric to a hyperbolic plane of curvature $\frac{1}{\cosh(r)}$, and its volume entropy is $h(S(r)) = \frac{h(0)}{\cosh(r)}= \frac{1}{\cosh(r)}$, hence the entropy is maximal if and only if $r=0$. 
For the general case, we are going to refined this argument showing that it is sufficient that a small part of $\Sigma$ is over $\Hyp^2$ for the entropy to be strictly less than $1$. 

Let $\Sigma$ be a embedded $\G$-invariant disk in $\Hyp^3$. We assume that $\Sigma \neq \D$, and we endowed $\Sigma$ with its induced metric. Let $x,y$ be two points on $\Sigma$. Let $c_\Sigma$ be a geodesic on $\Sigma$ linking $x$ to $y$. We parametrize $c_\Sigma$ by its Fermi coordinates, $(c,r)$. We then have 
\begin{eqnarray}\nonumber
d_{{\Sigma}} (x,y) &=& \int_0^L \| c'_{{\Sigma} }(t) \|_{\Sigma} dt \\ \nonumber
										 &=& \int_0^L \sqrt{r'(t)^2 +\cosh^2(r(t)) \| c' (t) \|_{g_0}^2 }dt. \\ \label{eq- evaluation distance}
										 &\geq & 	\int_0^L \cosh(r(t)) 		 \| c' (t) \|_{g_0} dt.
\end{eqnarray}	
We now endowed $\D$ with another distance than the one coming from hyperbolic metric. It will play the role of intermediary to compare  $d_{{\Sigma}} ( x,y) $ on $\Sigma$ with $d_{g_0}(s (x),s(y))$ on $\Hyp^2$. 

We call $\sigma$ the restriction of $s$ on $\Sigma$ . Since $\Sigma\neq \D$, there exists $x_0 \in \D\setminus \Sigma$, $\varepsilon >0$ and $\eta>0$ such that
$$d_{\Hyp^3} (\sigma^{-1} B(x_0, 2\varepsilon), \D) >\eta .$$ 
This means that \emph{all} the points in the pre-image of $B(x_0,2\varepsilon)$ by $\sigma$ are at distance greater than $\eta$ from  $\D$. 
We will assume that $2\varepsilon$ is smaller than the injectivity radius of $\Hyp^2/\G$ in order that the translations of $B(x_0,2\varepsilon)$ by $\G$ are disjoint. We have taken $2\epsilon$ in order to simplify the proof of Lemma \ref{lem -clef}. 

We now consider on  $\D$ the metric $g_m$ defined by putting weight on the translations of $B(x_0,2\varepsilon)$ by $\G$. 
\begin{definition}
We define $g_m$ by
$$g_m := \cosh(\eta)^2 g_0,$$
on $\G\cdot B(x_0,2\varepsilon)$.
and 
$$g_m := g_0,$$
elsewhere.
\end{definition}
We will index by  $m$ objects which depends on this metric. Remark that this metric is not continuous but it still defines a length space. Let $c : [0,1]\tv \D$ be a $C^1$ path  we then have $$\ell_m(c) = \int_0^1 \| \dot{c}(t)\|_{g_m} dt.$$
This gives a distance  $d_m$ on $\D$ by choosing :
 $$d_m(x,y) := \inf_{c} \{ \ell_m(c) \, | \, c(0)=x, \, c(1) =y\}.$$

In order to prove Theorem \ref{th - main fuchsian} we will compare the entropy of $(\D,d_m)$ with the one of $\Sigma$ and the one of $\Hyp^2$. The comparison with the entropy of $\Sigma$ is quite easy and follows quickly from the definition of $d_m$ and the inequality (\ref{eq- evaluation distance}). The comparison with the entropy of $\Hyp^2$ is more subtle. Indeed, there exist geodesics of $\Hyp^2$ which are geodesics for $(\D,d_m)$ (any lift of a closed geodesic which does not cross the ball $B(x_0,2\varepsilon)/\G$) on $\Hyp^2/\G$). We will first prove that two points of $\D$ which are joined by a geodesic of $\Hyp^2$ which crosses often $\G\cdot B(x_0,2\epsilon)$ are much farther  away from each other for $d_m$ distance, cf Lemme \ref{lem -clef}. Then, we will use a large deviation theorem for the geodesic flow (Theorem \ref{th large deviation sur les vecteurs}), to show that there are few geodesics which do not cross $\G\cdot B(x_0,2\epsilon)$ (Lemme \ref{lem il existe o tel que theta(Eo (R) est petit}). It will follow from these two results that the balls  of radius $R$ for $d_m$ are almost completely included in balls of radius $R/C$ of $\Hyp^2$ for $C>1$ (Lemma \ref{lem- delta dm < 1-delta}). The two comparisons give the proof of Theorem \ref{th - main fuchsian}. 

The comparison between $h(\Sigma)$ and the critical exponent of $(\D,d_m)$ follows from the inequality \ref{eq- evaluation distance} and the definition of $d_m$.

\begin{lemme}\label{lem - h(sigma)<delta(D,dm)}
We have $$h(\Sigma) \leq \delta((\D,d_m)).$$
\end{lemme}
\begin{proof}
Let $x\in \Sigma$ and $o=\sigma(x) \in \D$. Since $\Sigma/\G$  is compact, we have
$$h(\Sigma) = \lim_{R\tv \infty} \frac{1}{R}\log \Card\{ \g \in \G \, | \, d_\Sigma(\g x ,x )\leq R  \}. $$
And by definition 
$$\delta((\D,d_m)) = \lim_{R\tv \infty} \frac{1}{R}\log \Card\{ \g \in \G \, | \, d_m(\g o ,o )\leq R \} . $$
It is sufficient to prove that $d_\Sigma (x,y) \geq d_m(s(x),s(y))$, for all $x,y\in \Sigma$. Let $c_\Sigma=(c,r)$ be a geodesic on $\Sigma$ joining $x$ to $y$.  Recall that we have 
\begin{eqnarray*}
d_{{\Sigma}} (x,y)  &\geq & 	\int_0^L \cosh(r(t)) 		 \| c' (t) \|_{g_0} dt.
\end{eqnarray*}	
If $c(t)\notin \G \cdot B(x_0,2 \epsilon)$, then $\| c' (t) \|_{g_m}= \| c' (t) \|_{g_0}$. In particular  
$$\| c' (t) \|_{g_m} \leq \cosh(r(t)) 		 \| c' (t) \|_{g_0}.$$
If $c(t)\in \G \cdot B(x_0,2 \epsilon)$, then by definition of $g_m$, $\| c' (t) \|_{g_m}= \cosh(\eta)\| c' (t) \|_{g_0}$ and since $\Sigma$ is "far" from  $\D$, $r(t)>\eta$. In particular, 
$$\| c' (t) \|_{g_m} \leq \cosh(r(t)) 		 \| c' (t) \|_{g_0}.$$
Finally
\begin{eqnarray*}
d_{{\Sigma}} (x,y)  &\geq & 	\int_0^L \| c' (t) \|_{g_m} dt\\
							&\geq & l_m(c)\\
							&\geq & d_m(s(x),s(y)).
\end{eqnarray*}	

\end{proof}

Our next aim is to compare the distance $d_m$ and $d_{\Hyp^2}$.  Let us fix some notations before stating the first lemma. For all  $v\in T^1 \Hyp^2$,  let $\zeta^v_R $ be the probability measure on $T^1\Hyp^2$, defined for all borelian $E\subset T^1\Hyp^2$ by:
$$\zeta^v_R (E) =\frac{1}{R	}\int_0^R \chi _E \left( \phi_t^{\Hyp^2}(v) \right) dt$$
where $\chi_E $ is the indicator function of  $E$.  For a borelian $E$ which is a unitary tangent bundle of a subset of $\D$, $E:= T^1A$, we have:
$$\zeta^v_R(E) = \frac{1}{R} Leb \{ t\in [0,R] | c_v(t) \in A\}$$
since $ \phi_t^{\Hyp^2}(v)  \in E$ is equivalent to  $c_v(t) = \pi \phi_t^{\Hyp^2}(v) \in A$.

Let $L$ be the Liouville measure on the unitary tangent bundle of the quotient surface $T^1\Hyp^2/\G$. Recall that the metric $g_m$ is given by $g_m=\cosh^2(\eta)g_0$ on $T^1\G B(x_0,2\varepsilon)$. We fix the following $K := T^1( \G \cdot B(x_0,\varepsilon))$. \footnote{We use a ball of half the size, for a technical reason that appears at the beginning of the proof of Lemma \ref{lem -clef}}
\begin{definition}
Let  $\kappa>0$  be such $L(K/\G) -2\kappa>0 $.
We define the following sets,
$$\mathcal{E} (R) :=  \{ v \in T^1 \Hyp^2 \, |\,  \zeta^v_R (K) > L(K/\G) - \kappa \},$$
and for all points  $o\in \Hyp^2$, we note
$$\mathcal{E}_o (R) :=  \{ v \in T_o^1 \Hyp^2\,  |\,  \zeta^v_R (K) > L(K/\G) - \kappa \}.$$
\end{definition}
A geodesic of length $R$ whose direction is given by a vector $v\in \mathcal{E}(R)$ crosses $\pi K$ "often", that is at least a number proportional to $R$, cf. Figure \ref{fig K, Eo E}. Indeed, if $v\in \mathcal{E}(R)$ we have 
$$\frac{1}{R} \Leb \{ t \in [0,R] | c_0(t) \cap \pi K \neq \emptyset \} > L(K/\G)-\kappa >\kappa>0,$$
since  $\dot{c}_0(t) \in K$ is equivalent to   $c_0(t) \in \pi K$ by definition of $K$. 

\begin{figure}[h]
\begin{center}   \includegraphics[scale=0.4]{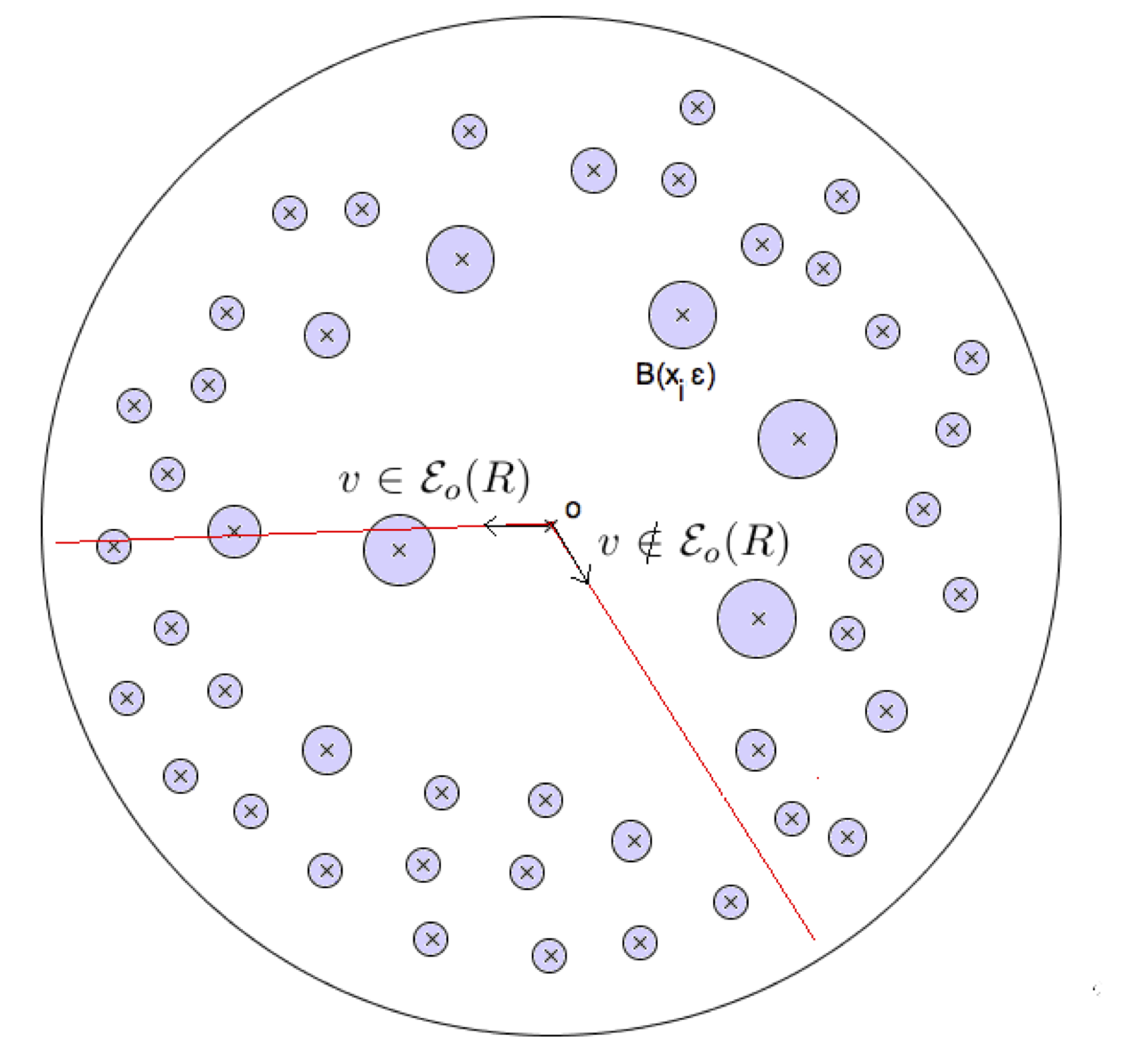}
    \caption{\label{fig K, Eo E} $\G\cdot B(x_0,\varepsilon)$, $\mathcal{E}_o(R)$ et $\mathcal{E}^c_o(R)$.}
\end{center}
\end{figure}

The next argument is the key in the proof of Theorem \ref{th - main fuchsian}. It shows that we can compare the length of a geodesic  in $\Hyp^2$ which crosses often $\pi K$ with its $d_m$ length. 

\begin{lemme}\label{lem -clef}
There exists $C>1$, such that for all $R>0$, for all $v\in \mathcal{E}_o (R) $ and for all  $x \in \{ \exp(tv) \, | \, t\in [R,2R]\}$, we have :
\begin{eqnarray}\label{eq -lem clef}
d_m(o,x) \geq  Cd_{\Hyp^2}(o,x).
\end{eqnarray} 
\end{lemme}

\begin{proof}
Let $c_0$ be the geodesic for $g_0$ and  $c_m$ be a minimizing geodesic for $g_m$ between $o$ and $x$. Let $d$ be the hyperbolic distance between $o$ and $x$, $d=d_{\Hyp^2}(o,x)$, and we parametrize $c_0$ by unit speed  we thus have $c_0(d) =x$. Let $N(R)$ be the number of intersections between $\pi K$ and $c_0([0,R])$, that is $N$ is the number of connected components of $c_0([0,R]) \cap \pi K$. On one hand, all  components of  $c_0([0,R]) \cap \pi K$ are inside balls of radius $\epsilon$, hence $c_0$ "stays" at most $2\epsilon$  in each components. On the other hand, the hypothesis $v\in  \mathcal{E}_o (R) $, implies that 
$$\frac{1}{R} Leb \{ t \in [0,R] | c_0(t) \cap \pi K \neq \emptyset \} > L(K/\G)-\kappa =\kappa>0.$$
These two facts imply that $2\epsilon N(R) \geq \kappa R$, that is to say 
\begin{equation}\label{eq -majoration N(R)}
N(R) \geq \frac{\kappa}{2\epsilon } R.
\end{equation}

For $i\leq N(R)$, let $t_i \in [0,d]$ such that $c_0(t_i) \in\pi K$ and $c_0[t_{i-1},t_i] \setminus \pi K $ is connected: we just have chosen a point $x_i = c_0(t_i)$ in each balls  of $\pi K$ crossing $c_0$. There exists $\g_i\in \G$ such that  $x_i \in B(\g_i x_0,\epsilon) $ hence $B(x_i,\epsilon) \subset B(\g_i x_0, 2\epsilon)$ on which the metric $g_m$ is  $g_m=\cosh^2(\eta) g_0$.  
See Figure \ref{fig balls}.
Therefore the geodesic $c_0$ is divided into $N(R)$ segments: $[x_i,x_{i+1}]$, such that for every $i$ we know that on the ball $B(x_i,\epsilon)$ the metric $g_m$ is given by $g_m=\cosh^2(\eta) g_0$.  We want a lower bound on $d_m(o,x)$, therefore we can estimate the length of $c_m$ with the metric given by $\cosh^2(\eta) g_0$ on the smaller balls $B(x_i,\epsilon)\subset B(\g_i x_0, 2\epsilon)$ and, $g_0$ on the rest of the plane. 

\begin{figure}[H]
\begin{center}   \includegraphics[scale=0.62]{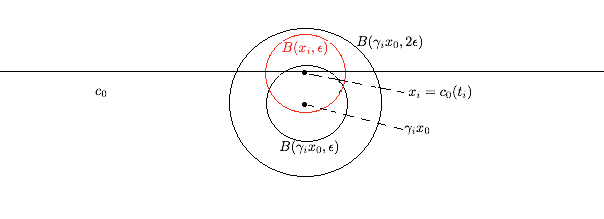}
    \caption{\label{fig balls} $c_0$ meets $B(\gamma_i x_0, \epsilon)$. $B(x_i, \epsilon)\subset B(\gamma_i x_0, 2\epsilon)$. }
\end{center}
\end{figure}

We call $y_i$ the middle of $[x_i,x_{i+1}]$. We now restraint our attention on one segment $[y_i,y_{i+1}]$. Let $0<a<1$ whose  dependence on  $\eta$ will be made clear in the rest of the proof. We are going to analyse two different cases. 
\paragraph{ Assume $c_m$ crosses $B(x_i,a \epsilon)$\\}  Let $\Delta_i$ be the lines (geodesics in $\Hyp^2$ )  orthogonal to $c_0$ and passing through $y_i$. Let $z^1_i$ and $z^2_i$ be the end points of the diameter of $B(x_i,\epsilon)$ defined by $z^1_i =c_0(t_i- \epsilon)$ and $z^2_i = c_0(t_i +\epsilon)$, and call $D^1_i$ and  $D^2_i$ the lines orthogonal to $c_0$ and passing through  $z^1_i$ and $z^2_i$. See figure \ref{fig cm rencontre B}. \\
We want to consider the intersections between $c_m$ and the lines $\Delta_i$, $D^1_i$ and   $D^2_i$. There might have many intersections. We will call first (resp. last) intersection of $c_m$ with a line $D$ the point $c_m(t_f)$ (resp $c_m(t_l)$) where $t_f := \inf\{t \, | c_m(t)\in D\}$ (resp $t_l := \sup\{t \, | c_m(t)\in D\}$).\\
Let $A'_i,B'_i$ and $C'_i$  be the last  intersections of $c_m$ with, respectively, $\Delta_i$, $D^1_i$   and $D^2_i$. Let $B_i,C_i$ and $A_{i+1}$ be the first intersections of $c_m$ with, respectively, $D^1_i$ ,  $D^2_i$ and $\Delta_{i+1}$. This divides $c_m$ in five connected components: 
$[A'_i,B_i]$, $[B_i,B_i']$, $[B'_i, C_i]$, $[C_i,C'_i]$, $[C'_i, A_{i+1}]$.

 Our work will be to give a lower bound for the  length of each components cf. Figure \ref{fig cm rencontre B}. Since it might happen that $B_i=B'_i$ and $C_i=C_i'$ the bound on the length of those two components will be trivial: $d_m(B_i,B'_i)\geq 0$ and $d_m(C_i,C'_i)\geq 0$.

The $g_m$-length of $c_m$ from $A'_i$ to  $B_i$ is equal (or larger) to its $g_0$-length  since the metric $g_m$ is equal to the metric $g_0$ outside $K$. Moreover the $g_0$-length of $c_m$ from $A'_i$ to $B_i$ is greater than $d_{g_0} (y_i,z^1_i)$  since the orthogonal projection decreases lengths. We then have  $$d_m(A'_i,B_i) \geq d_{g_0} (y_i,z^1_i).$$
 For the same reasons we have 
$$d_m(C'_i,A_{i+1}) \geq d_{g_0} (z^2_i,y_{i+1}).$$

 We want to give a lower bound for the  $g_m$-length of $c_m$ between $B'_i$ and $C_i$.  We made the assumption that  $c_m$ crosses the ball  $B(x_i,a\varepsilon)$ hence  $c_m$ stays at least $2\varepsilon -2a\varepsilon$ in the ball  $B(x_i,\varepsilon)$. In other words if $c_m$ is unitary  \emph{for $g_0$} we have $\Leb\{t \, | \, c_m(t) \cap B(x_i,\varepsilon) \neq \emptyset \} \geq 2\varepsilon -2a \varepsilon$.  In the ball $B(x_i,\varepsilon)$,  the metric $g_m$ is equal to   $\cosh(\eta)^2g_0$ hence the $g_m$-length   satisfies
 \begin{eqnarray*}
 d_m( B'_i,C_i) &\geq & \int_{\{t\,  | \, c_m(t) \cap B(x_i,\varepsilon) \neq \emptyset \} }\| \dot{c}_m(t) \|_m dt \\ 
						&= & \int_{\{t \, | \, c_m(t) \cap B(x_i,\varepsilon) \neq \emptyset \} }\cosh(\eta) \\ 
				 		&\geq & \varepsilon \cosh(\eta) (2 -2a).\\
\end{eqnarray*}  
Choose  $a>0$ such that
$\cosh(\eta) (2\varepsilon -2a\varepsilon)> 2\varepsilon $, that is to say $a\leq 1 -\frac{1}{\cosh  (\eta)}$. In order to fix the idea we set $a:=\frac{1}{2}(1 -\frac{1}{\cosh  (\eta)})$. This implies
\begin{eqnarray*}
d_m( B'_i,C_i) &\geq &\varepsilon \cosh(\eta) (2 -2a) \\
					 &=&\varepsilon \cosh(\eta) \left(2 -\left(1 -\frac{1}{\cosh  (\eta)}\right)\right) \\	
					& = & (\cosh(\eta) +1) \varepsilon\\
					&= & 2\varepsilon +  \varepsilon[\cosh(\eta) -1)]\\
					&=& d_{g_0} (z^1_i,z^2_i) + \varepsilon[\cosh(\eta) -1)].
\end{eqnarray*}
Finally we proved 
\begin{eqnarray}\label{eq - dm est plus grande que dhyp si cm rencontre B}
d_{m} (A_i,A_{i+1})\geq d_{m} (A'_i,A_{i+1}) \geq d_{g_0}(y_i,y_{i+1}) +\varepsilon[\cosh(\eta) -1].
\end{eqnarray}

\begin{figure}[H]
\begin{center}   \includegraphics[scale=0.4]{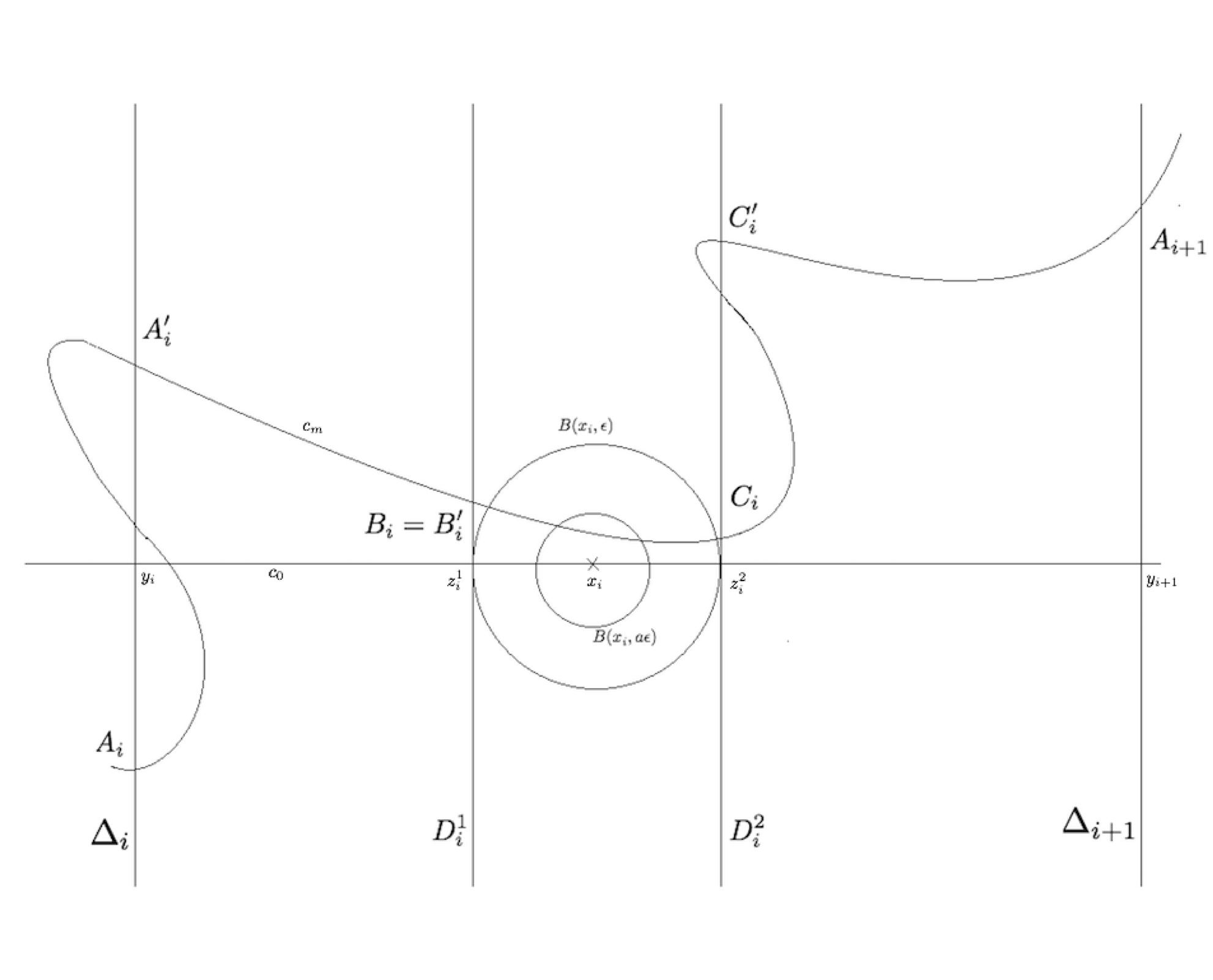}
    \caption{\label{fig cm rencontre B} $c_m$ crosses $B(x_i,a\varepsilon)$.}
\end{center}
\end{figure}

\paragraph{ Assume $c_m$  does not cross $B(x_i,a \varepsilon)$\\}
Let $\Delta_i$ be the line orthogonal to $c_0$ and passing through $y_i$ and $\Omega_i$ the one through $x_i$. Call $A'_i$ the last intersection of $c_m$ and $\Delta_i$ and $E_i$ the first intersecton of $c_m$ with $\Omega_i$.  Since $c_m$ does not cross $B(x_i a\epsilon)$, $E_i$ is in one of the connected component of $\Omega_i \backslash B(x_i, a\epsilon)$.  Named $e_i$ the intersection of $S(x_i, a\epsilon)$ (the sphere of center $x_i$ and diameter $a\epsilon$) and  $\Omega_i$ in the same connected component as $E_i$, this is also the orthogonal projection of $E_i$ on $B(x_i, a\epsilon)$. See figure \ref{fig c m ne renconrtre pas B}.\\
We parametrise the geodesic $\Omega_i$ by $\R$, we give $\omega : \R \tv \Hyp^2$ such that $\omega(\R) =\Omega_i$. We suppose that $\omega(0) =x_i$ and the orientation is chosen in order to have $\omega(a\epsilon) =e_i$. The function $t\tv d_{g_0} (\omega(t) ,\Delta_i)$ is convex, which has a minimum at $0$, it is hence increasing on $\R^+$. Therefore, $d_{g_0}(\Delta_i,E_i) \geq  d_{\Hyp^2}(\Delta_i,e_i)$. It follows that
$$d_m(A'_i,E_i) \geq d_{\Hyp^2}(A'_i,E_i)\geq d_{g_0}(\Delta_i,E_i) \geq  d_{g_0}(\Delta_i,e_i).$$
Let us compute $d_{g_0}(\Delta_i,e_i)$. We fix the following notations :  
\begin{eqnarray*}
L&=&d_{g_0}(\Delta_i,e_i)\\
l&=& d_{g_0}(y_i,x_i)\\
H&=&d_{g_0}(y_i,e_i)
\end{eqnarray*}
Now Pythagore's theorem in hyperbolic geometry for the triangle $(y_i x_i e_i)$ gives 
$$\cosh(l)\cosh(a\varepsilon) = \cosh(H).$$
Let $\theta$ be the angle $\widehat{x_iy_ie_i}$. We have 
$$\cos(\theta) = \frac{\tanh(l)}{\tanh(H)},$$
and
$$\sin(\pi/2 -\theta) = \frac{\sinh(L)}{\sinh(H)}.$$
Hence \begin{eqnarray*}
\sinh(L)& =& \sinh(H) \frac{\tanh(l)}{\tanh(H)}\\
			&=&		\cosh(H) \tanh(l)\\
			&=&\cosh(a\varepsilon)\sinh(l).
\end{eqnarray*}

From this equation, we \emph{cannot} conclude that $L>l+u $ for some $u>0$. Indeed if $L$ goes to $0$ so does $l$. To avoid this problem we are going to assume that $l$ is greater than the injectivity radius of $S$. 

Remark the following property of $\sinh$ which is a consequence of easy calculus. For all $x_0>0$ and  $\varpi>1$, there exists $u>0$, such that for all $x>x_0$, we have $\varpi \sinh(x) \geq \sinh(x+u)$. Now we can choose  $y_i$ on $c_0$ in order to have $d_{g_0} (x_i,y_i)\geq s/2$ where $s$ is the injectivity radius of $\Hyp^2/\G$. Consequently, applying the previous property with  $\varpi= \cosh(a\varepsilon)$ and $x_0=s/2$, there exists $u>0$ such that. 
$$\cosh(a\varepsilon)\sinh(l)\geq \sinh(l+u).$$
Since $\sinh$ is increasing we deduce that 
$$L \geq l+u.$$
Altogether, we show that there exists $u>0$ such that
$$d_m(A'_i,E_i) \geq d_{g_0}(y_i,x_i)+u.$$
By the same arguments we can show that 
$$d_m(E'_i,A_{i+1})  \geq  d_{g_0}(x_i,y_{i+1})+ u. $$
($E'_i$ is the last intersection of $c_m$ with $\Omega_i$). 
Hence, if $c_m$ does not meet $B(x_i,a \varepsilon)$, the $g_m$-length of $c_m$ between $A_i$ and $A_{i+1}$ satisfies, (taking trivial bounds for first and last intersections )
\begin{eqnarray}\label{eq -minoration dm si cm ne rencontre pas B(x,aepsilon)}
d_{m} (A_i,A_{i+1}) \geq d_{g_0}(y_i,y_{i+1})  +2u.
\end{eqnarray}

\begin{figure}[h]
\begin{center}   \includegraphics[scale=0.4]{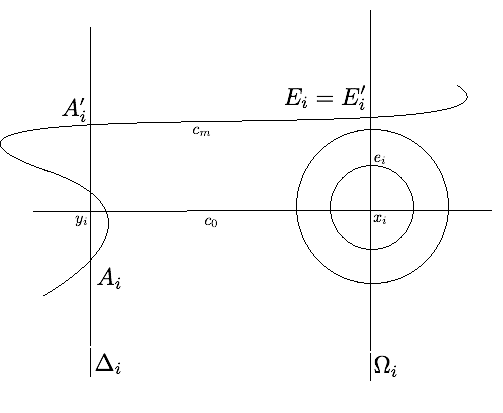}
    \caption{\label{fig c m ne renconrtre pas B} $c_m$ does not cross $B(x_i,a\varepsilon)$.}
\end{center}
\end{figure}

\paragraph{Conclusion}
Let $\alpha:=\min\{\varepsilon[\cosh(\eta) -1] ;2u\}$. From (\ref{eq - dm est plus grande que dhyp si cm rencontre B}) and (\ref{eq -minoration dm si cm ne rencontre pas B(x,aepsilon)}) we have  : 
$$d_{m} (A_i,A_{i+1}) \geq d_{g_0}(y_i,y_{i+1})   +\alpha.$$
Summing on  $i$ we get 
$$d_{m}( o,x) \geq d_{g_0} (o,x) +N(R) \alpha.$$
Equation (\ref{eq -majoration N(R)}) and the fact that $d_{g_0}(o,x)\leq 2R$ \footnote{this is where we use the upper bound on  $d_{g_0}(o,x)$.} imply that 
$$N(R)\geq \frac{\kappa}{2\varepsilon} R \geq \frac{\kappa}{4\varepsilon}  d_{g_0} (o,x). $$
Subsequently,  
$$d_{m}( o,x) \geq  \left( 1+ \frac{\alpha\kappa}{4\varepsilon} \right) d_{g_0} (o,x) .$$
This proves the Lemma with $C = \left( 1+ \frac{\alpha\kappa}{4\varepsilon} \right)$.
\end{proof}

We are now going to compare the entropy of  $(\D,d_m)$ with the one of $\Hyp^2$. 
Let us define $$\mathcal{F}_o(R)  =\{\exp(tv) \, | \, t\in \R^+, v\in \mathcal{E}_o(R)\}.$$
We note by  $B_m(o,2R)$ the ball of radius $2R$ for the $d_m$ distance. 
\begin{lemme}\label{lem Bm(o,2R) est inclus dans des boules hyperboliques plus petites}
Let $C' := \min(2,C)$ where $C$ satisfies the Lemma \ref{lem -clef}. We have for all  $o\in \D$, and all $R>0$ :
$$B_m(o,2R) \subset B_{\Hyp^2}\left(o,2R/C'\right)\cup  \Big( B_{\Hyp^2}(o,2R) \cap \mathcal{F}^c_o(R)\Big).$$
\end{lemme}

\begin{proof}
Indeed we have $B_m(o,2R) =\Big( B_m(o,2R) \cap\mathcal{F}_o(R) \Big) \cup \Big( B_m(o,2R)\cap \mathcal{F}^c_o(R) \Big)$.
Let $x\in B_m(o,2R) \cap\mathcal{F}_o(R)$. Since $d_{\Hyp^2} (o,x)\leq d_m(o,x)$, it follows that $d_{\Hyp^2}(o,x)\leq 2R$. There are only two possibilities. If $d_{\Hyp^2}(o,x) \leq R$, we have in particular $d_{\Hyp^2}(o,x) \leq \frac{2R}{C'}$.  However, if $d_{\Hyp^2}(o,x) \geq R$, we apply Lemma \ref{lem -clef} and we get  $d_{\Hyp^2}(o,x) \leq \frac{2R}{C} \leq \frac{2R}{C'}$. Therefore, 
$$B_m(o,2R) \cap\mathcal{F}_o(R)  \subset  B_{\Hyp^2}(o,\frac{2R}{C'}) \cap\mathcal{F}_o(R) \subset B_{\Hyp^2}(o,\frac{2R}{C'}). $$ 
Since we also have for $R>0$,  $B_m(o,2R)  \subset  B_{\Hyp^2}(o,2R)$, this gives
$$B_m(o,2R)\cap \mathcal{F}^c_o(R)  \subset B_{\Hyp^2}(o,2R) \cap \mathcal{F}^c_o(R),$$
and prove the lemma. 
\end{proof}
The  Liouville measure on $T^1\Hyp^2$ is the product of the riemannian measure of $\Hyp^2$ with the angular measure on every fiber. We denote this product by $L=d\mu(x)\times d\theta(x)$. Our aim is to show that the set  $\mathcal{E}^c_o (R)$ is small and the volume of  $\Big( B_{\Hyp^2}(o,2R) \cap \mathcal{F}^c_o(R)\Big)$ is small compared to the one of  $B_{\Hyp^2}(o,2R) $. For this we are going to use a large deviation theorem of Y. Kifer \cite{kifer1990large} which gives an upper bound on the mass of the vectosr which do not behave as the Liouville measure. 

Let $\mathcal{P}$ be the set of probability measures on  $T^1\Hyp^2/\G$ endowed with the weak topology. Let  $\mathcal{P}^t$ be the subset of $\mathcal{P}$  of probability measures invariant by the geodesic flow.   We also denote by $L$  the Liouville measure on the quotient $T^1\Hyp^2/\G$. Recall that for a vector $v\in T^1\Hyp^2/\G$ we denote  by $\zeta_v^R$ the  probability measure given for all borelians subset  $E\subset T^1\Hyp^2/\G$ by
$$\zeta^v_R (E) =\frac{1}{R	}\int_0^R \chi _E \left( \phi_t^{\Hyp^2/\G}(v) \right) dt.$$

\begin{theorem}\cite[Theorem 3.4]{kifer1990large}\label{th large deviation sur les vecteurs}
Let $\overline{A}$ be a compact subset of $\mathcal{P}$,
 $$\limsup_{T \tv \infty} \frac{1}{T} \log  L \left\{  v\in T^1\Hyp^2/\G\, |\, \,\zeta_v^T\in \overline{A}  \right\}  \leq -\inf_{\mu \in \overline{A}\cap \mathcal{P}^t} f(\mu) $$
where  $f(\mu) = 1- h_\mu (\phi_t ^{\Hyp^2/\G} ) $ and $h_\mu (\phi_t ^{\Hyp^2/\G}) $ is the entropy of the geodesic flow $\phi_t ^{\Hyp^2/\G}$ with respect to $\mu$. 
\end{theorem}
The fact that the theorem can be applied on this setting is explained after the Theorem 3.4 in \cite{kifer1990large}. In this reference the function $f$ is given by a formula which seems different.  One can look at \cite[Chapter 7]{paulin2015equilibrium}, where the authors explain in details why the geodesic flow of  negatively curved surfaces satisfies the hypothesis of Kifer's Theorem, and that one can take $f(\mu) = 1- h_\mu (\phi_t ^{\Hyp^2/\G} ) $.

\begin{lemme}\label{lem il existe o tel que theta(Eo (R) est petit}
There exists $o\in \Hyp^2$, $\alpha>0$ and $R_0>0$ such that for all  $R>R_0$ 
$$\theta_o(\mathcal{E}^c_o(R) )\leq e^{-\alpha R}.$$
\end{lemme}

\begin{proof}
Let us keep the notations of Lemma \ref{lem -clef}, $K=T^1\G\cdot B(x,\varepsilon)$ and we consider the following subset of $\mathcal{P}$
$$A := \{ \mu \in \mathcal{P} \, | \, \mu(K/\G) \leq L(K/\G)  -\kappa\}.$$
This set is \emph{not} closed for the weak topology. Its closure satisfies 
$$\overline{A} \subset \{ \mu \in \mathcal{P} \, | \, \mu(T^1\G \cdot B^{\circ} (x,\epsilon)/\G) \leq L(K/\G)  -\kappa\},$$
where $B^{\circ}(x,\epsilon)$ is the open ball. There might be equality between the two sets, but we won't use  it. \\
 However, since the unitary tangent bundle of the sphere $S(x,\epsilon)$ is transverse to the flow, we have : 
$$\{v\in T^1\Hyp^2/\G \, |\, \zeta_v^R \in A\} =\{v\in T^1\Hyp^2/\G \, |\, \zeta_v^R \in \overline{A}\}. $$
In other words, the measures $\zeta_v^R $ do not charge $T^1S(x,\epsilon)$.

Since $L\notin \overline{A}$ and $L$ is the unique measure of maximal entropy satisfying $h(L)=1$, we have
$$-\inf_{\mu \in \overline{A}} f(\mu) =-\alpha<0.$$
Besides, it is clear that the set $\mathcal{E}^c(R) = \{ v \in T^1 \Hyp^2 \, | \,\ \zeta^v_R (K)\leq   L(K/\G) - \kappa \}$  is $\G$-invariant from the  $\G$ invariance of $K$.   By definition and the previous remark we get 
\begin{eqnarray*}
\mathcal{E}^c(R)/\G & =& \left\{  v\in T^1\Hyp^2/\G \,| \,\, \zeta_v^R\in A  \right\}\\
								 & =& \left\{  v\in T^1\Hyp^2/\G \,| \,\, \zeta_v^R\in \overline{A}  \right\}.
\end{eqnarray*}
The Theorem \ref{th large deviation sur les vecteurs} says that there exists $R_0>0$ such that for all $R>R_0$ we have
$$L(\mathcal{E}^c(R)/\G) \leq e^{-\alpha R}.$$
The product structure of $L$ implies the existence of a point $o\in \Hyp^2/\G$ such that
$$\theta_o \left( \mathcal{E}_o^c(R)/\G) \right) \leq e^{-\alpha R}.$$
The Lemma follows, choosing any lift of $o$ in $\Hyp^2$. 
\end{proof}

We finish the proof of Theorem \ref{th - main fuchsian} with Lemma \ref{lem- delta dm < 1-delta}, which compare the critical exponent between $d_m$ and hyperbolic  distance. Lemmas \ref{lem - h(sigma)<delta(D,dm)} and  \ref{lem- delta dm < 1-delta}, conclude the proof. 
\begin{lemme}\label{lem- delta dm < 1-delta}
There exists $u>0$ such that 
$$\delta((\D,d_m)) \leq 1-u.$$
\end{lemme}

\begin{proof}
We are going to show that the volume entropy of $(\D,d_m)$ satisfies the inequality, that would imply the similar result on critical exponent. 

Let $o\in \D$ be a point satisfying Lemma \ref{lem il existe o tel que theta(Eo (R) est petit}. From Lemma \ref{lem Bm(o,2R) est inclus dans des boules hyperboliques plus petites}, we have
$$B_m(o,2R) \subset B_{\Hyp^2}(o,\frac{2R}{C'}) \cup  \Big( B_{\Hyp^2}(o,2R) \cap \mathcal{F}^c_o(R)\Big).$$
On one hand we have the classical upper bound   $\vol\left(B_{\Hyp^2}(o,\frac{2R}{C'})\right)=O(e^{2R/C'})$. On the other hand the volume form on  $\Hyp^2$ can be written in polar coordinates as $\sinh(r)drd\theta$, hence for all $R>R_0$ we get 
\begin{eqnarray*}
\vol\Big( B_{\Hyp^2}(o,2R) \cap \mathcal{F}^c_o(R)\Big) &= &\int_0^{2R} \int_{\mathcal{E}^c_o(R) } \sinh(r)d\theta dr\\
																						&\leq & \int_0^{2R} e^{-\alpha R} e^r dr\\
																						&\leq & e^{(2-\alpha)R}.
\end{eqnarray*}

Let $u >0$, defined by  $1-u= \max ( \frac{1}{C'}, (1-\alpha/2)) <1$. The last two upper bounds give
\begin{eqnarray*}
\vol(B_m(o,2R) ) &=&O(e^{2R/C'}) + O(e^{(2-\alpha)R})\\
							&=& O(e^{2(1-u)R})
\end{eqnarray*}
We finish by taking the $\log$ and the limit. 
\end{proof}

\section{Quasi-Fuchsian case}
\subsection{Geodesic intersection }\label{sec - geodesic intersection}
 Let $\Sigma$  be an incompressible surface in $M$. We designed by  $\phi_t^{\Hyp^3}$, $\phi_t^{\Sigma}$  the geodesic flows on the unitary tangent spaces  $T^1\Hyp^3$, $T^1\Sigma$ respectively. We named $\pi $ the projection from $T^1 \Hyp^3 $ to $\Hyp^3$.   The restriction of  $\pi$  to  $T^1 {\Sigma} $ will still be denoted by $\pi$.  There is two distances we can consider on ${\Sigma}$. The intrinsic one, defined as the infinimum of the length of curves staying on ${\Sigma}$ and the extrinsic one, where we take the distance in $\Hyp^3$. We will denote $d_\Sigma$ and $d$ this two distances.

First of all let us remark that there is no riemanniann metric on $\Sigma$ which induces $d$.  If such a metric existed, our Theorem \ref{th - main qf} would be a particular case of \cite{knieper1995volume}.
\begin{proposition}
If $\Sigma$ is not totally geodesic, there is no riemannian metric on $\Sigma$ which induces $d$. 
\end{proposition} 

\begin{proof}
Assume there is such a riemannian metric, named $g'$. Let $\epsilon>0$ be such that the exponential map for $g'$ is an embedding at every point. Let $c_{g'} : [0,\epsilon] \tv {\Sigma}$ be  a minimizing geodesic for $g'$ on ${\Sigma}$, then 
for all $t\in [0,\epsilon],$ 
$$d_{g'}(c_{g'}(0),c_{g'}(t)) +d_{g'}(c_{g'}(t),c_{g'}(\epsilon)) =d_{g'}(c_{g'}(0),c_{g'}(\epsilon))$$
But since we suppose that ${g'}$ induces $d$ we have the same equality for $d$
$$d(c_{g'}(0),c_{g'}(t)) +d(c_{g'}(t),c_{g'}(\epsilon)) =d(c_{g'}(0),c_{g'}(\epsilon))$$
and this implies that $c_{g'}$ is a geodesic for $\Hyp^3$. Hence  every points of $\Sigma$  is included in a totally geodesic disc, therefore $\Sigma$ is totally geodesic. 
\end{proof}

Consider the  following function
$$\fonction{a}{T^1{\Sigma}  \times \R}{\R}{(v,t)}{d(\pi \phi_t^{{\Sigma}}(v), \pi(v))}$$
Let $t_1,t_2 \in \R$ and $v\in T^1{\Sigma}  $, we have by the triangle inequality, 
\begin{eqnarray*}
a(v,t_1+t_2) &=& d(\pi \phi_{t_1+t_2}^{{\Sigma}}(v), \pi(v)) \\
					&\leq & d(\pi \phi_{t_1+t_2}^{{\Sigma}}(v), \pi \phi_{t_1}^{{\Sigma}}(v)) + d(\pi \phi_{t_1}^{{\Sigma}}(v),\pi(v))\\
					&\leq & d(\pi \phi_{t_2}^{{\Sigma}}( \phi_{t_1}v), \pi\phi_{t_1}^{{\Sigma}}(v)) + d(\pi \phi_{t_1}^{{\Sigma}}(v),\pi(v))\\
					&\leq & a(\phi^{{\Sigma}}_{t_1}v, t_2) + a(v,t_1)
\end{eqnarray*}

hence $a$ is a subadditive cocycle for the geodesic flow $ \phi_t^{{\Sigma}}$. Since $a$ is $\G$ invariant  it defines a subbadittive cocycle on $T^1\Sigma$, still denoted by $a$. 

The following is a consequence of Kingman's subadditive ergodic theorem \cite{kingman1973subadditive}.
\begin{theorem}
Les $\mu$ be a $\phi_t^{\Sigma}$ invariant probability measure on $T^1\Sigma$. Then 
$$I_\mu(\Sigma, M, v) := \lim_{t\tv \infty} \frac{a(v,t)}{t}$$
exists for $\mu$ almost $v\in T^1 \Sigma $ and defines a $\mu$-integrable function  on $T^1\Sigma$, invariant under the geodesic flow and  we have : 
$$\int_{T^1\Sigma} I_\mu (\Sigma, M,v)d\mu = \lim_{t\tv \infty } \int_{T^1\Sigma}\frac{a(v,t)}{t} d\mu.$$
Moreover if $\mu$ is ergodic $I_\mu(\Sigma,M,v)$ is constant $\mu$-almost everywhere. In this case, we write $I_\mu(\Sigma,M)$.
\end{theorem}

\subsection{Patterson Sullivan measures}
We called $\Lambda $ the limit set of $\G$ acting on $\Hyp^3$. 
Since $\G$ acts cocompactly on ${\Sigma}$, and on the convex core $C(\Lambda)$, the three geometric spaces $\G$ (seen as its Cayley graph), ${\Sigma}$ and $C(\Lambda)$ are quasi-isometric. We assume from now on that $(\Sigma,g)$ has negative curvature, hence there is a unique geodesic in each homotopy class of curves, and for every pair of points in ${\Sigma}$  there is a unique geodesic which joints them. Let $c_{{\Sigma}}$ be a geodesic on ${\Sigma}$, and denoted by  $c_{{\Sigma}} (\pm\infty)$ its limit points on $\Lambda$.  There is a unique $\Hyp^3$-geodesic $c_{\Hyp^3}$ whose endpoints are $c_{{\Sigma}}(\pm\infty)$. Since ${\Sigma}$ is quasi-isometric to $C(\Lambda)$, the two geodesics $c_{\Hyp^3}$ and $c_{{\Sigma}}$ are at bounded distance. 

Let $p\in {\Sigma}$ and call $pr_p^{{\Sigma}}$ the projection from ${\Sigma}$ to $\Lambda$ defined as follows. For any point $x \in {\Sigma}$ call $c_{p,x}^{{\Sigma}}$ the geodesic on ${\Sigma}$  which joint $p$ to $x$, then
$$pr_p^{{\Sigma}}(x) = c_{p,x} ^{{\Sigma}}(+\infty).$$

We will denote the equivalent projection in $\Hyp^3$ by $pr_p^{\Hyp^3}$. There is two small distinctions to notice between  $pr_p^{\Hyp^3}$ and $pr_p^{{\Sigma}}$. First $pr_p^{\Hyp^3} $ is defined for every points in $\Hyp^3$, whereas $pr_p^{{\Sigma}}$ is only defined for points in ${\Sigma}$. Second is that the codomain of $pr_p^{{\Sigma}}$ is exactly $\Lambda$ whereas the codomain of $pr_p^{\Hyp^3}$ is all $S^2$. 

As we just said, for all $\xi\in \Lambda$ the geodesics, $c_{p,\xi}^{{\Sigma}}$ and $c_{p,\xi}^{\Hyp^3}$ are at bounded distance, and this bound depends only on the quasi-isometry between ${\Sigma}$ and $C(\Lambda)$. There exists $C_1$ such that  for all $\xi \in \Lambda$ the  Hausdorff distance between geodesics $c_{p,\xi}^{{\Sigma}}$ and $c_{p,\xi}^{\Hyp^3}$ is less than $C_1$. \\
 Let $x \in {\Sigma}$, $R>0$ and consider the ball $B_{\Hyp^3}(x,R)$ in $\Hyp^3$ of center $x$ and radius $R$. Now take  $\xi \in pr_p^{\Hyp^3} (B(x,R-C_1) ) \cap \Lambda$, this means that the $\Hyp^3$-geodesic from $p$ to $\xi$ cross the ball  $B_{\Hyp^3}(x,R-C_1)$. This $\Hyp^3$-geodesic is at bounded distance $C_1$  of  the ${\Sigma}$-geodesic joining $p$ to $\xi$. Hence $c_{p,\xi}^{{\Sigma}}\cap \left( B_{\Hyp^3}(x,R) \cap {\Sigma } \right) \neq \emptyset$, this proves that $\xi \in pr_p^{{\Sigma}} (B_{\Hyp^3}(x,R) \cap {\Sigma })$. 

The same argument shows that 
$$pr_p^{{\Sigma}} (B_{\Hyp^3}(x,R) \cap {\Sigma }) \subset pr_p^{\Hyp^3} (B_{\Hyp^3}(x,R+C_1) ) \cap \Lambda \subset pr_p^{\Hyp^3} (B_{\Hyp^3}(x,R+C_1) ).$$

The distance on $\Sigma$  and on $\Hyp^3$ are locally equivalent: for every $R>0$ there exists $C_2$ such that all balls satisfy  the following
$$B_{ {\Sigma }} (x,R-C_2) \subset B_{\Hyp^3} (x,R) \cap \Sigma \subset B_{ {\Sigma }} (x,R+C_2)$$

Set $C =\max(C_1,C_2)$ we then have 

\begin{theorem}\label{comparaison des boules}
$$\begin{array}{ccccc}
& & pr_p^{ {\Sigma }} (B_{ {\Sigma }} (x,R-C) ) & & \\
& & \cap & & \\
pr_p^{ \Hyp^3} (B_{ \Hyp^3} (x,R-C) ) \cap \Lambda & \subset & pr_p^{ {\Sigma }} (B_{\Hyp^3} (x,R)  \cap {\Sigma } ) &  \subset &   pr_p^{ \Hyp^3} (B_{ \Hyp^3} (x,R+C) ) \\
& & \cap & & \\
 & & pr_p^{ {\Sigma }} (B_{ {\Sigma }}(x,R+C) ) &  &   \\
\end{array}
$$
\end{theorem}

Before proving Theorem \ref{th - main qf}, we will recall some basic facts about Patterson-Sullivan measure. Some classical references for this are the papers of Patterson and Sullivan themselves, \cite{patterson1976limit} and \cite{sullivan1979density}, the lecture of J-F. Quint \cite{quint2006overview} and the monograph of T. Roblin \cite{roblin2003ergodicite}. 
Let $(X,g)$ be a simply connected manifolds with negative curvature and   $X(\infty)$ its geometric  boundary. If $\G$ is a discrete group acting on $(X,g)$  we can associated a family of measures $\{\mu_p^g\}_{p\in X}$ on $X(\infty)$ constructed as follows. 
Let $x,y$ two points of $X$ and consider the Poincaré series:
$$P(s) := \sum_{\g\in \G} e^{-sd(\g x,y)} .$$
The convergence of $P(s)$ is independent of $x$ and $y$ by the triangle inequality. It converges for $s> \delta(\G)$ and diverges for $s<\delta(\G)$. If the action is cocompact, $\delta  (\G) = h(g)$ and the series diverges at $h(g)$.  
Then we define the probability measure 
$$\mu^g_{p,x} (s) := \frac{ \sum_{\g\in \G} e^{-sd(\g x,p)} \delta_{\g x}}{ \sum_{\g\in \G} e^{-sd(\g p,p)} }.$$
By compactness of the set of probability measures on $X(\infty)$, we obtain a measure on $X(\infty)$ by taking a weak limit of a sequence $\mu^g_{p,x} (s_n)$\footnote{It is a classical result of Sullivan that there is in fact a unique limit, up to normalization. It is equivalen tto the ergodicity of Bowen-Margulis measure \cite[Chapter 1]{roblin2003ergodicite}}
$$\mu^g_p := \lim_{s_n \tv  h(g)} \mu^g_p(s_n).$$
 It is supported on the accumulation points of $G$, that is to say the limit set. 
 
 These measures called \emph{Patterson-Sullivan measures} have the following properties. They  are quasi-conformal, ie. for all $p\in X$ and all $\xi, \eta\in \Lambda$, we have: 
 $$\frac{d\mu_p^g }{d \mu_q^g} (\xi) = e^{-h(g) \beta_\xi(p,q)},$$
 where $\beta_\xi(p,q) = \lim_{z\tv \xi} d_g(p,z) - d_g(q,z).$

They are also  $\G$-equivariant, ie. for all $\g \in \G$, and all $p\in X$, we have: 
$$\mu_p^g \circ \g = \mu_{\g ^{-1} p } ^g.$$

Moreover we know  these measures behave locally like $h(g)-$Hausdorff measures.  See \cite[Lemma 4.10]{quint2006overview} for example. 
\begin{lemme}[Shadow's lemma]\label{lemme de l'ombre}
For $R>0$ sufficiently large,  there exists $c>1$ such that  for all $x \in X$ 
$$\frac{1}{c} e^{-h(g) d_g(x,p)}  \leq \mu_p^g (pr_p^g (B_g(x,R)))\leq c e^{-h(g) d_g(x,p)}.$$
\end{lemme}

Suppose that $X/\G$ is compact, from Patterson-Sullivan measure, we can construct an invariant measure on $T^1X/\G$. Let $\Lambda^{(2)}$ be $\Lambda \times \Lambda \setminus diagonal$, there is a natural identification of $\Lambda^{(2)}\times \R $ and $T^1X$, a vector $v\in T^1 X$ is identified  with $(c_v (+\infty) , c_v (-\infty) , \beta_{c_v (+\infty) } (p,\pi v))$. The Bowen-Margulis measure is defined by 
$$d\mu_{BM} (\xi, \eta , t) =  e^{2h(g) \scal{\xi}{\eta}_p} d \mu_p^g(\xi) d \mu_p^g(\eta) dt $$
where $  \scal{\xi}{\eta}_p $ is the Gromov product:
$$\scal{\xi}{\eta}_p = \frac{1}{2} \left(  \beta_\xi(z,p)+  \beta_\eta(z,p) \right),$$ 
where $z$ is any point on the geodesic $(\xi,\eta)$. 

Let us recall the classical fact that the measure $\mu_{BM}$ is $\G-$invariant and define therefore a measure on $T^1X/\G$. Let $z\in (\xi, \eta)$, we have:
\begin{eqnarray*}
\scal{\g \xi}{\g \eta}_p & = &\frac{1}{2} \left(  \beta_{\g\xi}(\g z,p)+  \beta_{\g \eta}(\g z,p) \right)\\
									&=&\frac{1}{2} \left(  \beta_{\g\xi}(\g z, \g p)+ \beta_{\g\xi}(\g p,  p) +  \beta_{\g \eta}(\g z,\g p) +\beta_{\g\eta}(\g p, p) \right)\\
									&=&\frac{1}{2} \left(  \beta_{\xi}( z,  p)+ \beta_{ \eta}(z, p)+\beta_{\g\xi}(\g p,  p) +  \beta_{\g\eta}(\g p, p) \right)\\
									&=&\scal{\xi}{\eta}_p + \frac{1}{2} \left(  \beta_{\g\xi}(\g p,  p) +  \beta_{\g\eta}(\g p, p) \right).
\end{eqnarray*} 
By the quasi-conformal behaviour of $\mu_p^g$, we have:
\begin{eqnarray*}
e^{2h(g) \scal{\g \xi}{\g \eta}_p} d \mu_p^g(\g \xi) d \mu_p^g(\g \eta) & =&  e^{2h(g) \scal{\xi}{\eta}_p} e^{h(g)   \beta_{\g\xi}(\g p,  p))}d \mu_p^g(\g \xi) e^{h(g)   \beta_{\g\eta}(\g p,  p))} d \mu_p^g(\g \eta) \\
									&=& e^{2h(g) \scal{\xi}{\eta}_p} d \mu_p^g(\xi) d \mu_p^g(\eta).
\end{eqnarray*}

The invariance by the geodesic flow is clear by definition and it  is shown in \cite{nicholls1989ergodic} that $\mu_{BM}$ is ergodic. 

Finally we will need the following theorem, which is classical for compact manifolds endowed with two differents negatively curved metrics. Since we treat a case slightly different we give the proof. 

\begin{theorem}\label{L'eaglite des PS implique proportionalite des psectres}
If $\mu_p^{{\Sigma}} $ and $\mu_p^{\Hyp^3} $ are equivalent, then the marked length spectrum of $\Sigma$ is proportional to the marked length spectrum of $M$. 
\end{theorem}
Remark that in the Fuchsian case, any surface equidistant to the totally geodesic one has a metric proportional to $\Hyp^2$ and therefore satisfies the hypothesis of the Theorem. It seems likely that it is the only case where the length spectrum is proportional to the one of the ambiant manifold, however it is still unknown.

\begin{definition}\label{def - gromov distance}
For all $\xi,\eta \in \partial X^{(2)}$, we define the function $D_X$ by: 
$$D_X(\xi, \eta) =\exp(-\scal{\xi}{\eta}_p).$$
\end{definition}
It is shown in \cite{ghys1990espaces} that $D_X^a$ for $a>0$ small enough is a distance, called Gromov distance.  However, we do not need here such renormalisation. 


The proof of Theorem \ref{L'eaglite des PS implique proportionalite des psectres} is in two steps. The first one we prove that if the Patterson Sullivan measures are equivalent then the functions $D_\Sigma$ and $D_{\Hyp^3} $ are Hölder equivalent. The second one we prove that this last condition implies the proportionality of the length spectrum. 

\begin{lemme}\label{lem measure equivalent implies gromov distance holder}
If $\mu_p^{{\Sigma}} $ and $\mu_p^{\Hyp^3} $ are equivalent, then the $D_{\Hyp^3}$ and $D_{\Sigma}$ are Hölder equivalent. 
\end{lemme}

\begin{proof}
Let us consider on $\Lambda^{(2)}$ the Bowen-Margulis currents defined by 
$$\nu_\Sigma (\xi,\eta) = \frac{d\mu^p_\Sigma (\xi) d\mu^p_\Sigma(\eta)}{D_{\Sigma}(\xi,\eta) ^{2\delta(\Sigma)}}$$
$$\nu_{\Hyp^3} (\xi,\eta) = \frac{d\mu^p_{\Hyp^3} (\xi) d\mu^p_{\Hyp^3}(\eta)}{D_{\Hyp^3}(\xi,\eta) ^{2\delta(\Hyp^3)}}.$$
These two measures are $\G-$invariant by the previous computations made for the Bowen-Margulis measures. 

By assumption $\mu^p_\Sigma$ and $\mu_p^{\Hyp^3}$  are equivalent, therefore  $\nu_\Sigma$ and $\nu_{\Hyp^3}$ are also equivalent. The ergodicity and the  $\G$-invariance implies the existence of $c>0$ such that
$$\nu_\Sigma=c\nu_{\Hyp^3}.$$

Since $\mu_p^{{\Sigma}}$ and $\mu_p^{\Hyp^3} $ are equivalent there exists a function  $f :\Lambda \tv \R^+$ such that $\mu_p^{{\Sigma}} (\xi) =f(\xi) \mu_p^{\Hyp^3} $. 
We have
$$f(\xi)f(\eta) D^{\delta(\Hyp^3)}_{\Hyp^3}(\xi,\eta) = c D_\Sigma^{\delta(\Sigma)} (\xi,\eta).$$
We see that $f$ is equal almost everywhere to a continuous function. We can therefore suppose that $f$ is continuous on  $\Lambda$ hence strictly positive. By compacity, there exists $C>1$ such that $\frac{1}{C}\leq f(\xi) \leq C$. 
Finally we get what we stated 
$$\frac{c}{C^2} D_\Sigma^{\delta(\Sigma)} (\xi,\eta) \leq D^{\delta(\Hyp^3)}_{\Hyp^3}(\xi,\eta)  \leq C^2 c D_\Sigma^{\delta(\Sigma)} (\xi,\eta).$$
\end{proof}

We now show the second part 

\begin{lemme}\label{lem holder equivalent implique spectre marqué proportionnel}
If $D_\Sigma$ and $D_{\Hyp^3}$ are Hölder equivalent the marked length spectrum are proportional. 
\end{lemme}
\begin{proof}
In \cite[Section 3.5]{paulin2015equilibrium} the authors show that in a very general setting we have:
$$\lim_{n\tv \infty}\frac{1}{n} \log [g^-,g+,g^n(\xi),\xi]= \ell(g),$$
where $\ell(g)$ is the displacement of  $g$ and $[g^-,g^+,g^n(\xi),\xi] = \frac{D(g^-,g^n(\xi))D(g^+,\xi)}{D(g^-,\xi)D(g^+,g^n(\xi))}$.


In particular, we can apply this result to $\Sigma$ and $\Hyp^3$ we get 
$$\lim_{n\tv \infty}\frac{1}{n} \log [g^-,g^+,g^n(\xi),\xi]_{\Sigma}= \ell_\Sigma(g),$$
and
$$\lim_{n\tv \infty}\frac{1}{n} \log [g^-,g^+,g^n(\xi),\xi]_{\Hyp^3}= \ell_{\Hyp^3}(g).$$

By assumption on the distances $D_\Sigma,D_{\Hyp^3}$, there exists $C>1$ such that  we have
$$\frac{1}{C} [g^-,g^+,g^n(\xi),\xi]_{\Hyp^3}^r\leq [g^-,g^+,g^n(\xi),\xi]_{\Sigma}\leq C  [g^-,g^+,g^n(\xi),\xi]_{\Hyp^3}^r.$$
Hence
$$\ell_\Sigma (g) =r\ell_{\Hyp^3}(g).$$
\end{proof}

Theorem \ref{L'eaglite des PS implique proportionalite des psectres} follows directly from Lemmas \ref{lem measure equivalent implies gromov distance holder} and \ref{lem holder equivalent implique spectre marqué proportionnel}.

We will show at the very end of this article, that if $\Sigma$ has the same length spectrum as $M=\Hyp^3/\G$ then $\G$ is Fuchsian to prove Corollary \ref{cor - h(sigma)=delta imply fuchsian}. It might be also true even when we only suppose that they are proportional, however this does not follow from our proof. 
\subsection{Entropy comparison}
We finally get to the proof of Theorem \ref{th - main qf}. First we prove the inequality using the behaviour of Patterson-Sullivan measures and a volume comparison of a subset of $\Sigma$, the proof follows the same lines as \cite[Theorem 3.4]{knieper1995volume}. Then we prove the equality case using Theorem \ref{L'eaglite des PS implique proportionalite des psectres}.
\begin{theorem}
Let $\Sigma \subset \Hyp^3$ be a $\G$ invariant embedded disk, whose induced metric $g$  has negative curvature, then 
$$h(g) \leq I_{\mu_{BM}} (\Sigma , M) \delta(\G).$$
Moreover, the equality occurs if and only if the marked length spectrum of $\Sigma$ is proportional to the marked length spectrum of  $M$. In this case, the proportionality factor is given by $\ell_\Sigma(g) I(\Sigma, M) =\ell_M(g)$.
\end{theorem}

\begin{proof}
The geodesic flow is ergodic with respect to the Bowen-Margulis measure $\mu_{BM}$, hence for $\mu_{BM}$-almost all  $v\in T^1\Sigma$ we have : 
$$ \lim_{t\tv \infty} \frac{a(v,t)}{t}=I_\mu(\Sigma, M).$$
Let $v$ and $v'$ be two unit vectors on the same weak stable manifold. 
Then $d(c_{v'}(t),c_{v'}(0) ) \leq d(c_{v'}(t),(c_{v}(t)) +  d(c_{v}(t),(c_{v}(0)) + d(c_{v}(0),(c_{v'}(0))$, and the same inequality holds interchanging the role of $v$ and $v'$. Moreover  $d(c_{v'}(t),(c_{v}(t)) $ decreases exponentially since $v$ and $v'$ are on the same weak stable manifold. Hence $\lim_{t\tv \infty} \frac{a(v,t)}{t}$ exists if and only if $\lim_{t\tv \infty} \frac{a(v',t)}{t}$ does. 

Let $v_p(\xi)$ denotes the unitary vector in $T^1_p {\Sigma}$ such that $c_{v_p(\xi)} (\infty) =\xi$. The previous fact and the product structure of $d\mu_{BM}$ assures that for $\mu_p^g$ almost all $\xi \in \partial {\Sigma}$ we have 
$$\lim_{t\tv \infty} \frac{a(v_p(\xi),t)}{t}=I_\mu(\Sigma, M).$$

For all $\epsilon >0$ and $T>0$ we define the set 
$$A_p^{T,\epsilon} = \left\{ \xi \in  \partial {\Sigma} | \left|  \frac{a(v_p(\xi),t)}{t}-I_\mu(\Sigma, M) \right| \leq \epsilon, \quad t\geq T\right\}.$$
For all $d\in ]0,1[$ and all $\epsilon >0$, there exists $T>0$ such that $\mu_p^{{\Sigma}}(A_p^{T,\epsilon}) \geq d$. For $t>T$ consider the subset $\{c_{p,\xi}  (t) | \xi \in A_p^{T,\epsilon}\} \subset S_g(p,t)$ of the geodesic sphere of radius $t$ and center $p$ on ${{\Sigma}}$.

Choose $\{ B_{{\Sigma}}(x_i,R) | i\in I\}$  a covering of this subset of fixed radius $R>0$ such that $x_i \in S_{{\Sigma}}(p,t)$ and $B_{{\Sigma}}(x_i,R/4)$ are pairwise disjoint. 
Then, by the local behaviour of $\mu_p^{{\Sigma}}$, there exists a constant $c>1$, independent of $t $, such that 
$$\frac{1}{c} e^{-h(g) t}  \leq \mu_p^{{\Sigma}} (pr_p^{{\Sigma}} (B_{{\Sigma}}(x_i,R)))\leq c e^{-h(g) t}.$$
It is clear that $A_p^{T,\epsilon}  \subset \cup_{i\in I} pr_p^{{\Sigma}} (B_{{\Sigma}}(x_i,R))$ and therefore, 
$$d \leq  \mu_p^{{\Sigma}} \left(  \cup_{i\in I} pr_p^{{\Sigma}} (B_{{\Sigma}}(x_i,R))  \right)\leq \sum_{i\in I} \mu_p^{{\Sigma}} (pr_p^{{\Sigma}} (B_{{\Sigma}}(x_i,R))) \leq c \Card(I) e^{-h(g)t}.$$

Since $\Hyp^3/\G$ is convex cocompact, $C_Q(\Lambda)/\G$ is compact, where $C_Q(\Lambda) $ is  the $Q$ neighbourhood of the convex core of $\Lambda$.  Hence  for any $Q>0$, 
$$\delta( \G)  = \lim_{R\tv \infty} \vol(B_{\Hyp^3} (o,R) \cap C_Q(\Lambda)).$$
 
 Now take $Q$ sufficiently large such that ${\Sigma  }$ is inside $C_Q (\Lambda)$. There exists $K$ such that $B_{{\Sigma}} (x_i,R/4) \subset B_{\Hyp^3}(x_i,R+K) \cap C_Q(\Lambda)$.
 
From the definition of the set $A_p^{T,\epsilon} $ , we then have that the disjoint union $\cup_{i\in I} B_{{\Sigma}} (x_i,R/4) \subset B_{\Hyp^3}(p,t(I_{\mu_{BM}} (\Sigma, \Hyp^3) +\epsilon) +  R+K)  \cap C_Q(\Lambda)$. It follows that, 

\begin{eqnarray*}
e^{h(g) t } &\leq & \frac{c}{d}\Card (I)\\
				&\leq & \frac{c}{d V} \sum_{i\in I} vol_{\Hyp^3} (B_{\Hyp^3} (x_i ,R/4))\cap C_Q(\Lambda) )\\
				&\leq & \frac{c}{d V}  vol_{\Hyp^3} ( B_{\Hyp^3} (p , t(I_{\mu_{BM}} (\Sigma, \Hyp^3) +\epsilon) +  R+K) \cap C_Q(\Lambda))
\end{eqnarray*}
Hence
$$h(g)\leq \frac{1}{t} \left( \log \frac{c}{dV} + \log vol_{\Hyp^3}( B_{\Hyp^3} (p , t(I_{\mu_{BM}} (\Sigma, \Hyp^3) +\epsilon)) +  R+K) \cap C_Q(\lambda)) \right) $$
Taking the limit $t\tv \infty$ we get 
$$h(g) \leq (I_{\mu_{BM}} (\Sigma, \Hyp^3) +\epsilon) \delta(\G)$$
and we conclude since $\epsilon$ is arbitrary.
\end{proof}

For the proof of the equality case in Theorem \ref{th - main qf} we will use the result equivalent to  \cite[Corollary 3.6]{knieper1995volume} in our context, that is: 
\begin{lemme}\cite{knieper1995volume}
Let $p \in {\Sigma} $ and $\mu_p^g$ the Patterson-Sullivan measure with respect to $p$ and $g$, there exists a constant $L$ such that for $\mu_p^g$ almost all $\xi \in \partial {\Sigma}$ there is a seqence $t_n\tv \infty $ such that 
$$|d( p, \pi \phi_{t_n} ^{{\Sigma}}v_p(\xi) ) - I_{\mu_{BM}} (\Sigma, \Hyp^3) t_n | \leq L.$$
\end{lemme}

\begin{proof}
It follows from Lemma 3.5 of \cite{knieper1995volume}, that our lemma is true provided there exists a constant  $C>0$ such that, for all $t_1,t_2>0$ and all $v\in T^1\Sigma$,
$$a(v,t_1) +a(\phi_{t_1} ^{{\Sigma}} v, t_2 ) \leq C +a (v,t_1+t_2).$$ 

Let $v\in T^1{\Sigma}$ and $c_v^{{\Sigma}}$ be the geodesic on ${\Sigma}$ directed by $v$. Recall that there exists $C_1$ such that the ${\Hyp^3}$-geodesic from $\pi(v) $ to $c_v^{{\Sigma} }(t_1+t_2)$ is at  bounded distance  $C_1$ of $ c_v^{{\Sigma}}(t_1+t_2)$, independant of $t_1$ and $t_2$. 
The $\Hyp^3$-geodesic from $p$ to $c_v^{{\Sigma} }(t_1)$ and  the one from $c_v^{{\Sigma} }(t_1)$  to $c_v^{{\Sigma} }(t_1+t_2)$ are also at bounded distance $C_1$ of $c_v^{{\Sigma} }$.  
This implies the desired property with $C=2C_1$. 
\end{proof}

\begin{proof}[Equality case in \ref{th - main qf}]
Suppose that $h(g) = I_{\mu_{BM}}(\Sigma,\Hyp^3)\delta(\G)$. Choose a point $p \in {\Sigma}$ and  $\xi \in \Lambda$, set $y_n :=  \pi \phi_{t_n} ^{{\Sigma}}v_p(\xi)$. From the above lemma, for $\mu_p^{{\Sigma}}$ almost all  $\xi$ we have 
 $$|d( p, y_n ) - I_{\mu_{BM}} (\Sigma, \Hyp^3) t_n | \leq L.$$

Set $R>0$ a fixed constant, by local property of the Patterson-Sullivan measure on $\Hyp^3$, there is $c_1$ such that 
$$\frac{1}{c_1} e^{-\delta (\G) d(p,y_n)} \leq \mu_p^{\Hyp^3} ( pr_{ \Hyp^3} B_{\Hyp^3} (y_n,R) ) \leq c_1  e^{-\delta (\G) d(p,y_n)},$$
by Theorem \ref{comparaison des boules}
$$pr_{ \Hyp^3} (B_{ \Hyp^3} (x,R-C) ) \cap \Lambda  \subset  pr_{ {\Sigma }} (B_{\Hyp^3} (x,R)  \cap {\Sigma } )   \subset    pr_{ \Hyp^3} (B_{ \Hyp^3} (x,R+C) ).$$
 Hence there is a constant $c_2$  such that 
$$\frac{1}{c_2} e^{-\delta (\G) d(p,y_n)} \leq \mu_p^{\Hyp^3} ( pr_{{\Sigma }} B_{\Hyp^3} (y_n,R) \cap {\Sigma} ) \leq c_1  e^{-\delta (\G) d(p,y_n)}.$$

By the local property  of the Patterson-Sullivan measure on ${{\Sigma }} $, there is $c_3$ such that 
$$\frac{1}{c_3} e^{-h ({{\Sigma }} ) d_{{\Sigma }} (p,y_n)} \leq \mu_p^{{\Sigma }}  ( pr_{{\Sigma }} B_{{\Sigma }}  (y_n,R) ) \leq c_3  e^{-h({{\Sigma }} ) d_{{\Sigma }} (p,y_n)},$$
and by Theorem \ref{comparaison des boules}
$$pr_{ {\Sigma }} (B_{ {\Sigma }} (x,R-C) ) \subset   pr_{ {\Sigma }} (B_{\Hyp^3} (x,R)  \cap {\Sigma } )
\subset pr_{ {\Sigma }} (B_{ {\Sigma }}(x,R+C) ).$$
Hence there is $c_4$ such that 
$$\frac{1}{c_4} e^{-h ({{\Sigma }} ) d_{{\Sigma }} (p,y_n)} \leq \mu_p^{{\Sigma }}  ( pr_{{\Sigma }} B_{\Hyp^3}  (y_n,R) \cap {\Sigma }) \leq c_4  e^{-h({{\Sigma }} ) d_{{\Sigma }} (p,y_n)}.$$

From the choice of $y_n$ and since $h({\Sigma }) = I_{\mu_{BM}}(\Sigma,\Hyp^3) \delta(\G)$  
$$e^{-L} e^{-\delta(\G) d(p,y_n)} \leq e^{-h(g) d_{{\Sigma }} (p,y_n)} \leq e^{L   } e^{-\delta(\G) d(p,y_n)}.$$
Hence there is $c_5>0$ such that  
$$\frac{1}{c_5} e^{-\delta(\G)d (p,y_n)} \leq \mu_p^{{\Sigma }}  ( pr_{{\Sigma }} B_{\Hyp^3}  (y_n,R) \cap {\Sigma }) \leq c_5  e^{-\delta(\G) d(p,y_n)}.$$

Finally we have a constant $c_6$ such that 
$$  c_6  \leq \frac{\mu_p^{{\Sigma }}  ( pr_{{\Sigma }} B_{\Hyp^3}  (y_n,R) \cap {\Sigma }) }{\mu_p^{\Hyp^3} ( pr_{{\Sigma }} B_{\Hyp^3} (y_n,R) \cap {\Sigma} ) } \leq c_6.$$

Since  $pr_{{\Sigma }} (B_{\Hyp^3}  (y_n,R) \cap {\Sigma}) \tv \xi$  the measures $\mu_p^{{\Sigma }} $ and $\mu_p^{\Hyp^3}$ are equivalent. 
We conclude by Theorem \ref{L'eaglite des PS implique proportionalite des psectres}.
\end{proof}

We finish this article by the proof of Corollary \ref{cor - h(sigma)=delta imply fuchsian}: 

\begin{corollary*}[\ref{cor - h(sigma)=delta imply fuchsian}]
Under the assumptions  of Theorem \ref{th - main qf} we have 
$$h(\Sigma) \leq \delta(\G,\Hyp^3),$$ 
with equality if and only if  $\G$ is fuchsian and $\Sigma$ is the totally geodesic hyperbolic plane, preserved by $\G$. 
\end{corollary*}

\begin{proof}
The inequality is obvious. Suppose the equality occurs. Then by Theorem \ref{th - main qf}, we have that the length spectrum is proportional to the one of $\Hyp^3/\G$ and moreover that $I(\Sigma, M)=1$. In other words the two length spectra are equal. 

Since $\Sigma$ is embedded in $\Hyp^3$, we can prove that the equality between the spectra implies that $\Sigma$ is totally geodesic by the folowing argument:\\
Let $\g\in \G$, and consider $A$ its axis in $\Sigma$. Then for all $p\in A$ we have
$$\ell_\Sigma(\g) = d_\Sigma (\g p, p) \geq d_{\Hyp^3}(\g p, p)\geq \ell_{\Hyp^3} (\g).$$
Since the two spectra are equal, these inequalities are equalities. In particular, it implies that $p$ lies in the axis of $\g$ in $\Hyp^3$. Therefore $A$ is a geodesic of $\Hyp^3$.

Let $c$ be the closed geodesic on $\Sigma$ represented by $g$ and consider $c'$ any geodesic that intersects $c$. Let $g'$ be a representative of this closed geodesic such that the axis $A'$ of $g'$ on $\Sigma$ intersects $A$. By similar computations as before, we see that $A'$ is a geodesic of $\Hyp^3$. 

Since the two geodesic intersects the endpoints of $A$ and $A'$ are cocyclic on the boundary of $\Hyp^3$, and in particular bounds a copy of $\Hyp^2$ inside $\Hyp^3$.  By similar arguments for any element $g\in \G$ such that its axis $A_g$ intersects $A$ and $A'$ we see that $A_g$ is a geodesic of $\Hyp^3$ and therefore that $A_g$ is included in the copy of $\Hyp^2$. This last fact implies that $\Sigma$ is included, therefore equal, to this copy of $\Hyp^2$ and finishes the proof of the corollary. 
\end{proof}

\subsubsection{A remark on length spectrum rigidity.}
As we said in the introduction, the proof of the last corollary rises the following question: If a quasi-Fuchsian has the same length spectrum as a negatively curved surface, is it Fuchsian ? Or more generally, if the two length spectra are proportional does it implie that it is Fuchsian ? The later question seems to be not known even if we suppose that the surface has constant negative curvature equal $-1$, and the problem in general seems to be quite hard. 

We answer the case of constant negative curvature:
\begin{theorem*}[\ref{th - spectre proportionel et constant curvature implique fuchsian}]
Let $M$ be a quasi-Fuchsian manifold and $\Sigma$ a hyperbolic (in the sense that it has constant curvature equal $-1$) surface. Suppose that $M$ and $\Sigma$ have proportional length spectrum (ie. there exists $k\in \R^+$ such that for all $\g \in \G, \, \ell_M(\g)=k \ell_\Sigma(\g)$), then $M$ is Fuchsian, $k=1$ and $\Sigma$ is isometric to the totally geodesic surface in $M$. 
\end{theorem*}
In this case we cannot use the entropy argument that is used when we suppose the equality of the two spectra. Our proof is inspired by the work F. Dal'bo and I. Kim \cite{dalbo2000criterion} and base by the following Theorem of Y. Benoist:

\begin{theorem}\cite{benoist1997asymptotiques}
Let $G$ be a semi-simple linear connected Lie group. Let $\G<G$ be a Zariski dense subgroup. Then the limit cone is convex with non-empty interior. 
\end{theorem}
The limit cone is the smallest closed cone of a Cartan subspace of $\mathfrak{g}$ containing $\log(\lambda(\G))$ where $\lambda(\g) $ is the Jordan projection. 

\begin{proof}[Proof of Theorem \ref{th - spectre proportionel et constant curvature implique fuchsian}]
Consider $\G$ a surface group and $\rho_{QF}$ a quasi-Fuchsian representation into $\PSL_2(\mathbb{C})$ and $\rho_0$ a Teichmüller representation in $\PSL_2(\R)$. Consider the diagonal representation: 
$$\rho = (\rho_{QF},\rho_0) \, :\, \G \tv \PSL_2(\mathbb{C})\times \PSL_2(\R).$$

The group $\PSL_2(\mathbb{C})\times \PSL_2(\R)$ is a semi-simple linear connected Lie group of rank $2$. The Jordan projection of an element $(\g_1,\g_2)$ is given by $(\ell_{\Hyp^3} (\g_1), \ell_{\Hyp^2} (\g_2))$ where $\ell_X$ is the translation length in $X$. 

Therefore if the two representations have proportional length spectrum, then the limit cone of $\rho(\G)$ is a line, in particular it has empty interior. Using Benoist's Theorem we conclude that $\rho(\G)$ is not Zariski dense, which implies that $M$ is Fuchsian.  Therefore the length spectrum of $\Sigma$ is $k$ times the length sprectrum of the hyperbolic surface $\Sigma_0=\Hyp^2/\rho(\G)$. By Otal's Theorem \cite{otal1990spectre} we get that $(\Sigma, g) =(\Sigma_0, k^2g_\Hyp)$, hence since $\Sigma$ is hyperbolic, we have $k=1$ and $\Sigma=\Sigma_0.$. 

\end{proof}


\begin{thebibliography}{GdlH90}

\bibitem[Ben97]{benoist1997asymptotiques}
Yves Benoist. 
\newblock Propriétés asymptotiques of groupes linéaires.
\newblock{\em Geom. and Funct. Anal.} 7 (1997), no. 1, 1–47.


\bibitem[BCG95]{besson1995entropies}
G{\'e}rard Besson, Gilles Courtois, and Sylvestre Gallot.
\newblock Entropies et rigidit{\'e}s des espaces localement sym{\'e}triques de
  courbure strictement n{\'e}gative.
\newblock {\em Geometric and functional analysis}, 5(5):731--799, 1995.

\bibitem[Bow79]{bowen1979hausdorff}
Rufus Bowen.
\newblock Hausdorff dimension of quasi-circles.
\newblock {\em Publications Math{\'e}matiques de l'IH{\'E}S}, 50(1):11--25,
  1979.

\bibitem[BT00]{bridgeman2000length}
Martin Bridgeman and Edward~C Taylor.
\newblock Length distortion and the hausdorff dimension of limit sets.
\newblock {\em American Journal of Mathematics}, pages 465--482, 2000.



\bibitem[DK00]{dalbo2000criterion}
Françoise Dal'bo, Inkang Kim
\newblock A critierion of conjugacy for Zariski dense subgroups
\newblock {\em Compte Rendu Acad. Sci. Paris}  t. 330, Série I, p 647-650, 2000.


\bibitem[GdlH90]{ghys1990espaces}
{\'E}tienne Ghys and Pierre de~la Harpe.
\newblock Espaces m{\'e}triques hyperboliques.
\newblock In {\em Sur les groupes hyperboliques d'apres Mikhael Gromov}, pages
  27--45. Springer, 1990.

\bibitem[Ham02]{hamenstadt2002ergodic}
Ursula Hamenst{\"a}dt.
\newblock Ergodic properties of function groups.
\newblock {\em Geometriae Dedicata}, 93(1):163--176, 2002.

\bibitem[Kif90]{kifer1990large}
Yuri Kifer.
\newblock Large deviations in dynamical systems and stochastic processes.
\newblock {\em Transactions of the American Mathematical Society},
  321(2):505--524, 1990.
  
\bibitem[Kin73]{kingman1973subadditive}
John F. C. Kigsman
\newblock Subadditive Ergodic Theory
\newblock {\em Ann. Probability} Vol. 1 (6): 883--889, 1973. 


  
  

\bibitem[Kni95]{knieper1995volume}
Gerhard Knieper.
\newblock Volume growth, entropy and the geodesic stretch.
\newblock {\em Math. Res. Lett}, 2:39--58, 1995.

\bibitem[Nic89]{nicholls1989ergodic}
Peter~J Nicholls.
\newblock {\em The ergodic theory of discrete groups}, volume 143.
\newblock Cambridge University Press, 1989.


\bibitem[Otal90]{otal1990spectre}
Jean-Pierre Otal, 
\newblock Le spectre marqué des longueurs des surfaces à courbure négative
\newblock {\em Ann. of Math.} (2) 131 (1990), no 1, 151-162.




\bibitem[Pat76]{patterson1976limit}
Samuel~J Patterson.
\newblock The limit set of a fuchsian group.
\newblock {\em Acta mathematica}, 136(1):241--273, 1976.

\bibitem[PPS15]{paulin2015equilibrium}
Fréderic Paulin, Marc Pollicott and Barbara Schapira
\newblock Equilibrium states in negative curvature.
\newblock {\em Astérisque}, volume 373, 2015.

\bibitem[Qui06]{quint2006overview}
JF~Quint.
\newblock An overview of patterson-sullivan theory.
\newblock In {\em Workshop The barycenter method, FIM, Zurich}, 2006.
\url{https://www.math.u-bordeaux.fr/~jquint/publications/courszurich.pdf}

\bibitem[Rob03]{roblin2003ergodicite}
Thomas Roblin.
\newblock Ergodicit{\'e} et {\'e}quidistribution en courbure n{\'e}gative.
\newblock {\em M{\'e}moire de la Soci{\'e}t{\'e} math{\'e}matique de France},
  (95):A--96, 2003. 
  
\bibitem[Sul79]{sullivan1979density}
Dennis Sullivan.
\newblock The density at infinity of a discrete group of hyperbolic motions.
\newblock {\em Publications Math{\'e}matiques de l'IH{\'E}S}, 50:171--202,
  1979.


\end{thebibliography}
\end{document}